\documentclass[11pt]{article}

\usepackage{amssymb}
\usepackage{amsmath}

\newtheorem{theorem}{Theorem}

\newenvironment{proof}{\mbox{\sc Proof:}}{$\Box$}
\newenvironment{proof*}{\mbox{\sc Proof:}\hspace{0.3 em}{\bf (*)}\hspace{0.3 em}}{$\Box$}

\renewcommand{\Re}{\mathbb{R}}

\begin{document}

\title{A Note on Absolutely Continuous Processes}
\author{Lars Tyge Nielsen \\ Department of Mathematics \\ Columbia University}
\date{January 2019}
\maketitle

\begin{abstract}
Every adapted absolutely continuous process has a predictable density. 
The set of adapted absolutely continuous processes equals the set of time integrals of progressive or predictable pathwise locally integrable processes.
\end{abstract}

\section{Introduction}

By analogy to continuous processes, a stochastic process is \emph{absolutely continuous} its paths are absolutely continuous---every path is the integral of some locally integrable ``density function''.

This paper shows that if an absolutely continuous process is adapted, then the density functions of the individual paths can be fitted together to form a predictable process.
In other words, the absolutely continuous process is the time integral of a predictable and pathwise locally integrable process.

Conversely, it is also true that the time integral of a predictable (or progressive) pathwise locally integrable process is adapted and absolutely continuous. Hence, the set of
adapted and absolutely continuous processes is exactly equal to the set of time integrals of predictable (or progressive) pathwise locally integrable processes.

The result is close to
Chung and Williams \cite[1990, Lemma 3.11 (ii)]{Chung-Williams:90}\index{Chung, K. L.}\index{Williams, R. J.}
and Letta \cite[1988, Proposition 3.1]{Letta:88}, and
the proof is almost the same, but spun to a different conclusion.

Chung and Williams show that if $\hat{a}$ is a jointly measurable and pathwise locally integrable process, then its time integral is an adapted process, and there exists a predictable process $a$ such that $\hat{a} = a$ almost everywhere (which will imply that the time integrals of $a$ and $\hat{a}$ are indistinguishable).

The difference is that we do not assume $\hat{a}$ to be jointly measurable.
Thus it is not relevant to ask whether the processes $\hat{a}$ and $a$ are equal almost everywhere. 
Instead, we assume that the time integral of $\hat{a}$ is adapted and show directly that it is identical to the time integral of $a$.

\section{Background Information}

Let $(\Omega,\mathcal{F})$ be a measurable space and let $F = (\mathcal{F}_{t})_{t \in [0,\infty)}$ be a filtration.
There is literally no probability measure involved in any of the concepts in this paper.

Let $\mathcal{B}(\Re)$, $\mathcal{B}([0,\infty))$, and $\mathcal{B}([0,t])$, $t \in [0,\infty)$ denote the Borel sigma-algebras on $\Re$, $[0,\infty)$, and $[0,t]$, respectively.

A stochastic \emph{process} is a mapping $X: \Omega \times [0,\infty) \rightarrow \Re$ such that for 
every $t \in [0, \infty)$, $X(t): \omega \mapsto X(t)(\omega) = X(\omega,t)$ is measurable with respect to $\mathcal{F}$ and $\mathcal{B}(\Re)$.

All processes in this paper are understood to be one-dimensional.
The generalization to higher-dimensional processes is trivial.

The process $X$ is \emph{measurable} if it is measurable with respect to $\mathcal{F} \otimes \mathcal{B}([0,\infty)$, and it is \emph{adapted} if for 
every $t \in [0, \infty)$, $X(t)$ is measurable with respect to $\mathcal{F}_{t}$.

A stochastic process $X$ is \emph{progressive} if
its restriction to
$\Omega \times [0,t]$ is measurable with respect to
$\mathcal{F}_{t} \otimes {\mathcal{B}}([0,t])$, for every $t \in [0,\infty)$.

Let $\mathcal{M}$ be the smallest
sigma-algebra with respect to which all progressive processes are measurable.
Then a process is progressive if and only if it is measurable with respect to $\mathcal{M}$.

Let $\mathcal{P}$ be the sigma-algebra on $\Omega \times [0,\infty)$ generated by the adapted and continuous processes.
Call it the \emph{predictable} sigma-algebra.
A process is \emph{predictable} if it is measurable with respect to $\mathcal{P}$.

Every predictable process is progressive, and every progressive process is measurable and adapted. Expressed differently,
$\mathcal{P} \subset \mathcal{M} \subset \mathcal{F} \otimes \mathcal{B}([0,\infty))$.

A function $f: [0,\infty) \rightarrow \Re$ is \emph{locally integrable} if it is measurable with respect to
$\mathcal{B}([0,\infty))$ and $\mathcal{B}(\Re)$ and
\[ \int_{0}^{t} |f| \, ds < \infty \]
for all $t \in [0,\infty)$.

A process $a$ is \emph{pathwise locally integrable} if every path of $a$ is locally integrable. In other words, for every $(\omega,t) \in \Omega \times [0,\infty)$,
\[ \int_{0}^{t} |a(\omega,s)| \, ds < \infty \]

If a process is progressive (in particular, if it is predictable) and pathwise locally integrable, then its time integral is an adapted (and absolutely continuous) process. Adaptedness follows from the Tonelli-Fubini theorem.

A function $F: [0,\infty) \rightarrow \Re$ is \emph{absolutely continuous} if there exists a locally integrable function $f: [0,\infty) \rightarrow \Re$ such that
\[ F(t) = \int_{0}^{t} f \, ds \]
for all $t \in [0,\infty)$. If so, then $F$ is differentiable at almost every $t \in (0,\infty)$ with $F'(t) = f(t)$.

A stochastic process $X$ is \emph{absolutely continuous} if every path of $X$ is absolutely continuous.

If $a$ is a progressive process such that
all paths of $a$ are locally integrable, then the mapping $X: \Omega \times [0,\infty) \rightarrow \Re$ defined by
\[ X(\omega,t) = \int_{0}^{t} a(\omega,s) \, ds \]
for all $(\omega,t) \in \Omega \times [0,\infty)$, is an adapted and, of course, absolutely continuous process.

\section{The Theorem}

\begin{theorem}
\label{adconv2-t}
If $X$ is an adapted and absolutely continuous process, then there exists a predictable process $a$ such that
all paths of $a$ are locally integrable and such that
\[ X(\omega,t) = \int_{0}^{t} a(\omega,s) \, ds \]
for all $(\omega,t) \in \Omega \times [0,\infty)$.
\end{theorem}
\begin{proof}
Since $X$ is absolutely continuous, there exists a mapping $\hat{a}: \Omega \times [0,\infty) \rightarrow \Re$ such that for each $\omega \in \Omega$, the mapping $t \mapsto \hat{a}(\omega,t)$
is locally integrable and
\[ X(\omega,t) = \int_{0}^{t} \hat{a}(\omega,s) \, ds \]
for all $(\omega,t) \in \Omega \times [0,\infty)$.

For each $h > 0$, define a mapping $X(h): \Omega \times [0,\infty) \rightarrow \Re$ by
\[ X(h)(\omega,t) =  \left\{ \begin{array}{cc} X(\omega,t)/h & \mbox{ if } 0 \leq t < h  \\
(X(\omega,t) - X(\omega,t-h))/h & \mbox{ if } 0 < h \leq t \end{array} \right. \]
Since $X$ is an adapted process, so is $X(h)$.
For each $\omega \in \Omega$, $X(h)(\omega,t)$ is continuous on $[0,h)$ and on $[h,\infty)$, and 
$X(h)(\omega,t) \rightarrow X(\omega,h)/h = X(h)(\omega,h)$ as $t \rightarrow h$, $t < h$. Hence, $X(h)$ is a continuous process. Since it is adapted and continuous, it is predictable.

Let $(h_{n})$ be a sequence of positive numbers such that $h_{n} \rightarrow 0$ as $n \rightarrow \infty$. 
Since $X(h_{n})$ is predictable for each $n$, the set
\[ A = \{ (\omega,t) \in \Omega \times (0,\infty) : X(h_{n})(\omega,t) \mbox{ converges} \} \]
is in the predictable sigma-algebra $\mathcal{P}$.

Define $a: \Omega \times [0,\infty) \rightarrow \Re$ by
\[ a(\omega,t) =  \left\{ \begin{array}{cc} \lim_{n} X(h_{n})(\omega,t) & \mbox{ if } (\omega,t) \in A \\
0 & \mbox{ otherwise } \end{array} \right. \]
Then $a$ is a predictable process.

Given $\omega \in \Omega$, the mapping $t \mapsto X(\omega,t): (0,\infty) \rightarrow \Re$ is differentiable at almost all $t$, with derivative $\hat{a}(t)$. Hence,
for almost all $t > 0$, $X(h_{n})(\omega,t) \rightarrow \hat{a}(\omega,t)$
as $n \rightarrow \infty$. This implies that $a(\omega,t) = \hat{a}(\omega,t)$ for almost all $t \in [0,\infty)$.

In particular, all paths of $a$ are locally integrable, and 
\[ X(\omega,t) = \int_{0}^{t} \hat{a}(\omega,s) \, ds  = \int_{0}^{t} a(\omega,s) \, ds \]
for all $(\omega,t) \in \Omega \times [0,\infty)$.
\end{proof}

\bibliographystyle{plain}
\bibliography{bookrefs,bookrefsadd,consolidated}

\end{document}